\documentclass[reqno,11pt]{amsart}

\numberwithin{equation}{section}

\usepackage{verbatim} % hacer comentarios
\usepackage{xspace} % dejar espacios
\usepackage{amssymb}
\usepackage{amsmath}
\usepackage{amsthm} %
\usepackage[latin1]{inputenc}
\usepackage{graphicx}
\usepackage{amscd,amssymb,euscript,graphics}
\makeindex

\newtheorem{theorem}{Theorem}[section]
\newtheorem{prop}[theorem]{Proposition}
\newtheorem{lem}[theorem]{Lemma}

\newtheorem{df}[theorem]{Definition}

\newcount\theTime
\newcount\theHour
\newcount\theMinute
\newcount\theMinuteTens
\newcount\theScratch
\theTime=\number\time
\theHour=\theTime
\divide\theHour by 60
\theScratch=\theHour
\multiply\theScratch by 60
\theMinute=\theTime
\advance\theMinute by -\theScratch
\theMinuteTens=\theMinute
\divide\theMinuteTens by 10
\theScratch=\theMinuteTens
\multiply\theScratch by 10
\advance\theMinute by -\theScratch

\def\today{{\number\day\space
 \ifcase\month\or
  January\or February\or March\or April\or May\or June\or
  July\or August\or September\or October\or November\or December\fi
 \space\number\year}}

\begin{document}

\title[Random Regular Graphs are not Gromov Hyperbolic]{Random Regular Graphs are not Asymptotically Gromov Hyperbolic}

\author{Gabriel H. Tucci} 
\thanks{G. H. Tucci are with Bell Laboratories, Alcatel-Lucent, 600 Mountain Avenue, Murray Hill, New Jersey 07974, USA. 
E-mail: gabriel.tucci@alcatel-lucent.com}

\begin{abstract}
In this paper we prove that random $d$--regular graphs with $d\geq 3$ have traffic congestion of the order $O(n\log_{d-1}^{3}(n))$ where $n$ is the number of nodes and geodesic routing is used. We also show that these graphs are not asymptotically $\delta$--hyperbolic for any non--negative $\delta$ almost surely as $n\to\infty$.
\end{abstract}

\maketitle

\section{Introduction and Motivation}
Gromov hyperbolicity is important not only in group theory, coarse geometry, differential geometry \cite{bridson, Gromov} but also in many applied fields such as communication networks \cite{10, NS1}, cyber security \cite{jonck} and statistical physics \cite{11}. Hyperbolicity is observed in many real world networks such as the Internet \cite{jonck, 9} and data networks at the IP layer \cite{NS1}. 

The study of traffic flow and congestion in graphs is an important subject of research in graph theory. Furthermore, it is an extremely important topic in network theory and more specifically, in the study of communication networks. One fundamental problem is to understand the traffic congestion under geodesic routing. More precisely, let $\{G_{n}\}_{n=1}^{\infty}$ be a family of graphs where $|G_{n}|=n$. For each pair of nodes in $G_{n}$, consider a unit flow that travels through the minimum path between nodes. Hence, the total traffic flow in $G_{n}$ is equal to $n(n-1)/2$. If there is more than one minimum path for some pair of nodes then the flow splits equally among all the possible geodesic paths. Given a node $v\in G_{n}$ we define $T_{n}(v)$ as the total flow generated in $G_{n}$ passing through the node $v$. In other words, $T_{n}(v)$ is the sum off all the geodesic paths in $G_{n}$ which are carrying flow and contain the node $v$. Let $M_{n}$ be the maximum vertex flow across the network
$$
M_{n}:=\max \Big\{T_{n}(v)\,:\,v\in G_{n}\Big\}.
$$
It is easy to see that for any graph $n-1\leq M_{n}\leq n(n-1)/2$. 

It was observed in many complex networks, man-made or natural, that the typical distance between the nodes is surprisingly small. More formally, as a function of the number of nodes $n$, the average distance between nodes scales typically as $O(\log n)$. Moreover, many of these complex networks, specially communication networks, have high congestion. More precisely, there exists a small number of nodes called the {\it core} where most of the traffic pass through, i.e. $M_{n}=\Theta(n^2)$.

We said that a family of graphs is asymptotically $\delta$--hyperbolic if there exists a non--negative $\delta$ such that for all $n$ sufficiently large the graph $G_{n}$ is $\delta$--Gromov hyperbolic (see \cite{Gromov}). It was observed experimentally in \cite{NS1}, and proved formally in \cite{Tucci1}, that if the family is asymptotically hyperbolic then the maximum vertex congestion scales as $\Theta(n^2)$. In Section \ref{congestion}, we show that for random $d$--regular graphs with $d\geq 3$ the maximum vertex congestion scales as $O(n\log_{d-1}^{3}(n))$, suggesting a non--hyperbolic nature. Furthermore, in Section \ref{non-hyper}, we show that this is indeed the case. More precisely, we show that random $d$--regular graphs are not $\delta$--hyperbolic for any non--negative $\delta$ asymptotically almost surely. This is in contrast with the well--known fact that these graphs are very good expanders and hence, they have a large spectral gap. 

\begin{figure}[htp]
\centering
\includegraphics[width=12cm]{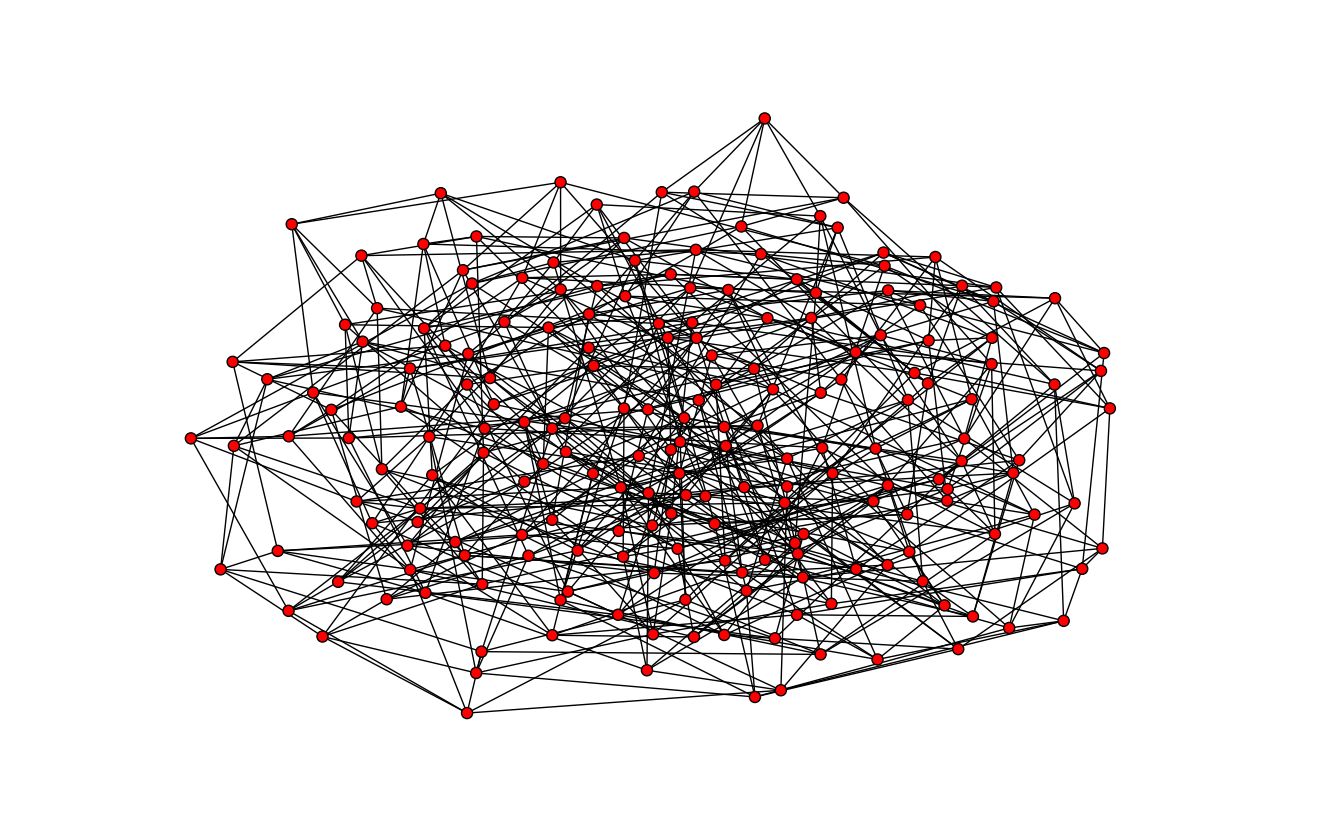}
\caption{A random $6$--regular graph with 200 nodes.}
\label{fig:6_200}
\end{figure}

\section{Maximum Vertex Congestion for Random $d$--regular Graphs}
\label{congestion}

\noindent For every pair of nodes in a graph $G$ there exists at least one shortest path (we assume that each link has unit length) between them. We denote a geodesic between a pair of nodes $a$ and $b$ by $[ab]$. A triangle $abc$ in $C$ is called $\delta$--thin
if each of the shortest paths $[ab]$, $[bc]$ and $[ca]$ is contained
within the $\delta$ neighborhoods of the other two. More
specifically,
\begin{equation}
[ab] \subseteq \mathcal{N}([bc],\delta) \cup \mathcal{N}([ca],\delta). 
\end{equation}
and similarly for $[bc]$ and $[ca]$.
A triangle $abc$ is $\delta$--fat if $\delta$ is the smallest $\delta$
for which $abc$ is $\delta$--thin. The notion of (coarse) Gromov
hyperbolicity (see \cite{Gromov}) is then defined as follows.
\begin{df}
A metric graph is $\delta$--hyperbolic if all geodesic
triangles are $\delta$--thin, for some fixed non-negative $\delta$.
\end{df}
\noindent It is clear that all tree graphs are $\delta$--hyperbolic, with $\delta = 0$. We also observe that all finite graphs are $\delta$--hyperbolic
for large enough $\delta$, e.g., by letting $\delta$ to be equal to the diameter of the graph.

In this Section, we explore the maximum vertex congestion with geodesic routing for random $d$--regular graphs. As discussed in the introduction, Gromov hyperbolic graphs have congestion of the order $\Theta(n^2)$. In particular, any $k$--regular tree (also called a Bethe lattice) has highly congested nodes. More precisely, the following result was proved in \cite{Tucci1}.
\begin{prop}
For a $k$--regular tree with $n$ nodes then 
\begin{equation}
M_{n} = \frac{k-1}{2k}\cdot(n-1)^2+n-1.
\end{equation}
Furthermore, the maximum congestion occurs at the root.
\end{prop}
Note that trees are some of the most congested graphs one can consider. The reason is that much of the traffic must pass through the root of the tree. 

In what follows we prove the main result of this Section but we first need the following results.
\begin{lem}\label{LL}
Let $G$ be a graph with bounded geometry, i.e. $\sup_{v\in G} \mathrm{deg}(v)\leq \Delta<\infty$. Then for every $v\in G$, the traffic flow passing through $v$ satisfies
\begin{equation}
T(v)\leq \Delta^2(\Delta-1)^{D-2}\cdot D^2
\end{equation}
where $D=\mathrm{diam}(G)$.
\end{lem}

\begin{proof}
Let $v\in G$ and define $S_{k}:=\{x\in G: d(v,x)=k\}$. Then it is clear that $G =\{v\}\cup\bigcup_{p=1}^{D}S_{p}$ and moreover $|S_{k}|\leq \Delta(\Delta-1)^{k-1}$. Furthermore,
$$
T(v)\leq \sum_{k+l\leq D}{|S_{k}|\cdot|S_{l}|}
$$
where the inequality is coming from the fact that if $k+l>D$ then the geodesic path between a node in $S_{k}$ and a node in $S_{l}$ does not pass through $v$. Hence,
$$
T (v)\leq \sum_{k+l\leq D}{\Delta^2(\Delta-1)^{k+l-2}}\leq \Delta^2(\Delta-1)^{D-2}D^2.
$$
\end{proof}

It was proved in \cite{BVega} that the diameter of a random $d$--regular satisfies the following upper bound almost surely.
\begin{theorem}\label{Boll}
For $d\geq 3$ and sufficiently large $n$, a random $d$--regular graph $G$ with $n$ nodes has diameter at most  
$$
D(G)\leq \log_{d-1}(n) + \log_{d-1}(\log_{d-1}(n))+C
$$
where $C$ is a fixed constant depending on $d$ and independent on $n$.
\end{theorem}

Now we are ready to prove the main Theorem of this Section.

\begin{theorem}
Let $G$ be a random $d$--regular graph with $d\geq 3$ and $n$ nodes. Then the maximum vertex flow with geodesic routing on $G$ is smaller than
\begin{equation}
M_n\leq d^{C}n \log_{d}^3(n) + o(n\log_{d}^3(n)).
\end{equation}
asymptotically almost surely as $n\to\infty$.
\end{theorem}

\begin{proof}
Let $v$ be a vertex in $G$ then by applying Lemma \ref{LL} we obtain
$$
T(v) \leq d^{D(G)}D(G)^2.
$$
Since
$$
D(G)\leq \log_{d-1}(n) + \log_{d-1}\log_{d-1}(n)+ C
$$
with high probability, we see that
\begin{eqnarray*}
T(v) &\leq & (\log_{d-1}(n) + \log_{d-1}\log_{d-1}(n)+ C)^{2}d^{\log_{d-1}(n) + \log_{d-1}\log_{d-1}(n)+ C}\\
& = & d^{C}n\log_{d-1}^3(n) + o(n\log_{d-1}^3(n)).
\end{eqnarray*}
Now taking the maximum over all the vertices finished the argument.
\end{proof}

In figure \ref{plot}, we observe the maximum vertex flow for a random $6$--regular graph as a function of the number of vertices. We also see in comparison the function $n^2$ and $n\log(n)$. 
\begin{figure}[htp]
\centering
\includegraphics[width=15cm]{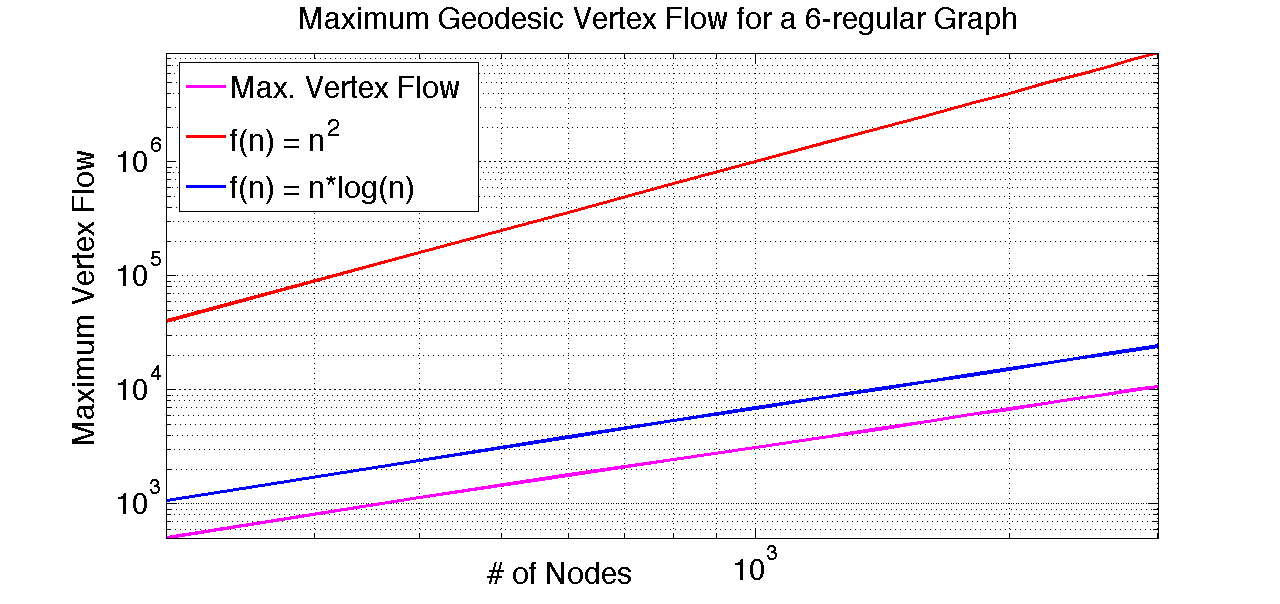}
\caption{Maximum vertex flow for a random $6$--regular graph as a function of the number of nodes $n$ averaged over $20$ realizations for each $n$.}
\label{plot}
\end{figure}

\section{Non--hyperbolicity for Random $d$--regular Graphs}
\label{non-hyper}

\noindent The previous definition of $\delta$--hyperbolicity is equivalent to Gromov's four points condition (see \cite{Gromov}). A graph $G$ satisfies the Gromov's four points condition, and hence it is $\delta$--hyperbolic, if and only if 
\begin{equation}\label{four}
d(x_1,x_3)+d(x_2,x_4)\leq \max\{d(x_1,x_2)+d(x_3,x_4), d(x_1,x_4)+d(x_2,x_3)\} + 2\delta
\end{equation}
for all $x_{1},x_{2},x_{3}$ and $x_{4}$ in $G$. 

\noindent As it was defined in \cite{itai}, an {\em almost geodesic cycle} $C$ in a graph $G$ is a cycle in which for every two vertices $v$ and $w$ in $C$, the distance $d_{G}(u,v)$ is at least $d_{C}(u,v)-e(n)$. Here $d_{G}$ is the distance in the graph and $d_{C}$ is the corresponding distance within the loop $C$. The following result was proved in \cite{itai}.

\begin{theorem}
Let $G$ be a random $d$--regular graph with $n$ nodes and let $\omega(n)$ be a function going to infinity such that $\omega(n)=o(\log_{d-1}\log_{d-1}n)$. Then almost all pair of vertices $v$ and $w$ in $G$ belong to an almost geodesic cycle $C$ with $e(n) = \log_{d-1}\log_{d-1} n + \omega(n)$ and $|C|=2\log_{d-1}n + O(\omega(n))$. 
\end{theorem}

\noindent Now we are ready to prove the main Theorem of this Section.

\begin{theorem}
Let $G$ be a random $d$--regular graph with $d\geq 3$ and $n$ nodes. Then for every non--negative $\delta$ the graph $G$ is not $\delta$--hyperbolic asymptotically almost surely.
\end{theorem}

\begin{proof}
By the previous theorem there exists and almost geodesic cycle $C$ with $e(n) = \log_{d-1}\log_{d-1} n + \omega(n)$ and $|C|=2\log_{d-1}n + O(\omega(n))$. Let $x_1,x_2,x_3,x_4\in C$ such that roughly $d_{C}(x_1,x_2)=d_{C}(x_2,x_3)=d_{C}(x_3,x_4)=d_{C}(x_4,x_1)=|C|/4$. Let $\gamma_{pq}$ be a geodesic path joining the points $x_{p}$ and $x_{q}$. Then by construction, 
$$
d_{C}(x_p,x_q)-e(n)\leq d_{G}(x_p,x_q)\leq d_{C}(x_p,x_q)
$$
for all $x_q$ and $x_q$. Hence,
\begin{equation*}
|C|/2-2e(n)\leq d_{G}(x_1,x_2)+d_{G}(x_3,x_4)\leq |C|/2
\end{equation*}
and 
\begin{equation*}
|C|/2-2e(n)\leq d_{G}(x_2,x_3)+d_{G}(x_1,x_4)\leq |C|/2.
\end{equation*}
Therefore,
\begin{equation}
\max\{d(x_1,x_2)+d(x_3,x_4), d(x_1,x_4)+d(x_2,x_3)\}\leq |C|/2=\log_{d-1}n + O(\omega(n))/2.
\end{equation}
On the other hand, it is clear that
\begin{equation*}
|C|-2e(n)\leq d_{G}(x_1,x_3)+d_{G}(x_2,x_4)\leq |C|.
\end{equation*}
Therefore,
\begin{equation}
d_{G}(x_1,x_3)+d_{G}(x_2,x_4)\geq 2\log_{d-1}n - 2\log_{d-1}\log_{d-1}n + O(\omega(n)).
\end{equation}
Thus,
\begin{eqnarray*}
d_{G}(x_1,x_3)+d_{G}(x_2,x_4) &-& \max\{d(x_1,x_2)+d(x_3,x_4), d(x_1,x_4)+d(x_2,x_3)\}  \\
& \geq & \log_{d-1}n - 2\log_{d-1}\log_{d-1}n - O(\omega(n))\to\infty.
\end{eqnarray*}
This four points violate the four points condition for all $\delta\geq 0$. Hence, it concludes the proof.
\end{proof}

\vspace{0.3cm}
\noindent {\it Acknowledgement}. This work was funded by NIST Grant No. 60NANB10D128.

\end{document}